\documentclass[a4paper,11pt,reqno]{amsart}
\usepackage[centertags]{amsmath}
\usepackage{amsfonts,amssymb,amsthm} 
\usepackage{esint}
\usepackage[dvips]{graphicx} 
\usepackage[english]{babel}
\usepackage{newlfont}
\usepackage{color}
\usepackage[body={16cm,23cm},centering]{geometry}

\newtheorem{theorem}{Theorem}
\newtheorem{lemma}{Lemma}

\newtheorem{corollary}{Corollary}

\theoremstyle{remark}
\newtheorem{remark}{Remark}

\newcommand{\R}{\mathbb{R}} 
\newcommand{\Sf}{\mathbb{S}}
\newcommand{\Ha}{\mathcal{H}}
\newcommand{\eps}{\varepsilon}
 
\newcommand{\la}{\langle}
\newcommand{\ra}{\rangle}
\newcommand{\unit}{\omega}
\mathchardef\emptyset="001F

\title[]{Robustness of the Gaussian concentration inequality\\
and the Brunn-Minkowski inequality}
\author[M. Barchiesi and V. Julin]
{M. Barchiesi and V. Julin}
\address[M. Barchiesi]{Universit\`{a} di Napoli ``Federico II'',
Dipartimento di Matematica e Applicazioni,
Via Cintia, Monte Sant'Angelo, I-80126 Napoli, Italy}
\email{barchies@gmail.com}
\address[V. Julin]{University of Jyv\"{a}skyl\"{a},
Department of Mathematics and Statistics,
P.O.Box 35 (MaD) FI-40014, Finland}
\email{vesa.julin@jyu.fi}
\thanks{This work has been supported by the  FiDiPro project ''Quantitative Isoperimetric Inequalities'' 
and the Academy of Finland grant 268393.}
\date{April 2, 2017}

\begin{document}
\maketitle

\begin{center}
\begin{minipage}{12.3cm}
\small{
\noindent {\bf Abstract.} 
We provide a sharp quantitative version of the Gaussian concentration inequality: 
for every $r>0$, the difference between the measure of the $r$-enlargement of a given set
and the $r$-enlargement of a half-space controls the square of the measure of the symmetric 
difference between the set and a suitable half-space. 
We also prove a similar estimate in the Euclidean setting for the enlargement with a general convex set.
This is equivalent to the stability of the Brunn-Minkowski
inequality for the Minkowski sum between a convex set and a generic one.

}
\end{minipage}
\end{center}

\bigskip

\section{introduction}

\noindent
In recent years there has been an increasing interest in the stability of concentration type inequalities 
(see \cite{Ch, CM16, EK,  FJ, FJ2, FMM, FigMP09}). In this paper we establish sharp stability estimates 
for the Gaussian concentration inequality and the Brunn-Minkowski inequality.

The Gaussian concentration inequality is one of the most important examples of concentration of measure phenomenon,
and a basic inequality in probability.  It states that the measure of the $r$-enlargement of a set $E$ is 
larger than the measure of the $r$-enlargement of a half-space $H$ having the same volume than $E$. 
Moreover, the measures are the same only if $E$ itself is a half-space. 
We recall that the $r$-enlargement of a given set $E$ is the Minkowski 
sum between the set and the ball of radius~$r$.  We refer to \cite{Le,Le2} for an introduction to the subject. 

A natural question is the stability of the Gaussian concentration inequality: 
can we control the distance between $E$ and $H$ with the gap of the Gaussian concentration 
inequality (the difference of the measures of the enlargements of $E$ and $H$)?   
We measure the distance between $E$ and $H$ by the Fraenkel asymmetry which is the measure of their symmetric difference. 
We prove that the gap of the Gaussian concentration inequality controls the square of the Fraenkel asymmetry.
This extends the stability of the Gaussian isoperimetric inequality \cite{BBJ}.
Our proof is based on the simple observation that the $r$-enlargement of a half-space $H$ can never completely 
cover the $r$-enlargement of the set $E$  (see Lemma \ref{lemma gauss}). This fact, which is essentially due to 
the convexity of the half-space, enables us to directly relate the problem to 
the stability of the Gaussian isoperimetric inequality. 

Our approach can be also adapted to the Euclidean setting for the enlargement with a given 
convex set $K$.  The Euclidean concentration inequality can be written as the Brunn-Minkowski inequality 
in the case of the Minkowski sum between a convex set and a generic one. Our interest in refining the Brunn-Minkowski 
inequality is motivated by the fact that it is one of the most fundamental inequalities  in analysis. 
We refer to the beautiful monograph \cite{G} for a survey on the subject. 
As for the Gaussian concentration inequality, also in the Euclidian case  we prove that the concentration gap controls 
the square of the asymmetry. 
As a corollary we obtain the sharp quantitative Brunn-Minkowski inequality when one of the set is convex. 

In order to state our result on the Gaussian concentration  more precisely, we introduce some notation. 
Throughout the paper we assume $n\geq2$.
Given a measurable set $E\subset\R^n$, its Gaussian measure is defined as
\begin{equation*}
\gamma(E):=\frac{1}{(2\pi)^{\frac{n}{2}}}\int_E e^{-\frac{|x|^2}{2}}dx.
\end{equation*}
Moreover, given $\unit \in\Sf^{n-1}$ and $s\in\R$,  $H_{\unit,s}$ denotes the half-space 
\begin{equation*}
H_{\unit,s}:=\{x\in\R^n \text{ : } \langle x,\unit\rangle<s\},
\end{equation*}
while $B_r$ denotes the open ball of radius $r$ centered at the origin.
We define also the function $\phi:\R\rightarrow(0,1)$ as the Gaussian measure of $H_{\unit,s}$, i.e.,
\begin{equation*}
\phi(s):=\frac{1}{\sqrt{2\pi}}\int_{-\infty}^s e^{-\frac{t^2}{2}}dt.
\end{equation*}
The concentration inequality states that, given a set $E$ with mass $\gamma(E)=\phi(s)$, 
for any $r>0$ one has
\begin{equation}\label{concentration ine}
\gamma(E+B_r)\geq \phi(s+r),
\end{equation}

\vspace{3pt}
\noindent
and the equality holds if and only if $E=H_{\unit,s}$ for some $\unit \in\Sf^{n-1}$. 
We have used the notation
\begin{equation*}
E+ B_r = \{ x+y \colon x \in E, \,\, y \in B_r\}
\end{equation*} 
for the $r$-enlargement of the set $E$.  In other words  $E+ B_r$ is the set of all  points
 which distance to  $E$ is less than $r$.
In order to study the stability of  inequality \eqref{concentration ine} we introduce the \emph{Fraenkel asymmetry}, 
which measures how far a given set  is   from a half-space. 
Given a measurable set $E$ with $\gamma(E)=\phi(s)$ we define
\begin{equation*}
\alpha_\gamma(E):=\min_{\unit\in\Sf^{n-1}}\gamma(E\triangle H_{\unit,s}),
\end{equation*}
where $\triangle$ stands for the symmetric difference between sets.

Here is our result for the stability of the Gaussian concentration.
\begin{theorem}\label{main theorem}
There exists an absolute constant $c>0$ such that for every $s\in\R$, $r>0$, and for every set $E\subset\R^n$ 
with $\gamma(E)=\phi(s)$ the following estimate holds:
\begin{equation}\label{main estimate}
\gamma(E+B_r)-\phi(s+r)\geq c\, \left(e^{s^2}e^{-\frac{(|s| +r+ 4)^2}{2}}\right)\,r\,\alpha_\gamma^2(E).
\end{equation}
\end{theorem}
The result is sharp in the sense that  $\alpha_\gamma^2(E)$ cannot be replaced by any other  function of $\alpha_\gamma(E)$ 
converging to zero more slowly.   
Previously in~\cite[Theorem 1.2]{CM16} a similar result was proved with $\alpha_\gamma^4(E)$ on the right-hand side.  
Another important  feature of \eqref{main estimate} is that the dimension of the space does not appear in the inequality. 
Finally we remark that since the left-hand side of \eqref{main estimate} converges to zero 
as $r$ goes  to infinity, so the right-hand side has to do. However, we do not known 
the optimal dependence on $s$ and $r$, or how they are coupled.

Recently, a different asymmetry has been proposed in \cite{EL}:
\begin{equation*}
\beta(E):=\min_{\omega\in\mathbb{S}^{n-1}}\big|b(E)-b(H_{\omega,s})\big|,
\end{equation*}
where $b(E):=\int_E x \, d\gamma$ is the (non-renormalized) barycenter of the set $E$. 
We call $\beta(E)$ \emph{strong asymmetry} since it controls the Fraenkel one (see \cite[Proposition 4]{BBJ}).
It would be interesting to replace in \eqref{main estimate} the Fraenkel asymmetry with this stronger one.

\vspace{8pt}
Moving on the Euclidean setting,  we assume  $K\subset\R^n$ to be an open, bounded, and
convex set which contains the origin. The Euclidean concentration inequality 
states that for a measurable set $E$ with $|E| = |K|$ it holds
\begin{equation} \label{aniso concentration}
|E+ rK| \geq |(1+r)K|
\end{equation}
for every $r >0$. Note that since $K$ is convex, it holds $K + rK = (1+r)K$. This is a special case of 
the  Brunn-Minkowski inequality which states that for given two measurable, bounded and 
non-empty sets $E,F\subset\R^n$ such that also $E+F:=\{x+y \colon x\in E, \,y\in F\}$ is measurable, it holds
\begin{equation} \label{brunn min orig}
|E+F|^{1/n}\geq |E|^{1/n} + |F|^{1/n}.
\end{equation}
The concentration inequality \eqref{aniso concentration} follows from the Brunn-Minkowski inequality by choosing $F = rK$. 
However, when $F$ is convex then \eqref{aniso concentration} is equivalent to \eqref{brunn min orig}.  
We define the Fraenkel asymmetry of a set $E$ with respect to $K$ as the quantity
\begin{equation*} 
\alpha(E) = \inf_{x \in \R^n} |(E +x) \triangle sK)|, \qquad \text{where } \, s=(|E|/|K|)^{1/n}.
\end{equation*}

Here is our result for the stability of the Euclidian concentration.
\begin{theorem}\label{main thm 2}
There exists a dimensional constant $c_n>0$ such that for every $r>0$ and for every set $E\subset\R^n$ 
with $|E| = |K|$ the following estimate holds:
\begin{equation}\label{conce convex}
|E+ rK| - |(1+r)K| \geq c_n \, \max\{r^{n-1},r\}\, \frac{\alpha^2(E)}{|E|}. 
\end{equation}
\end{theorem}
Also in this case the quadratic exponent on $\alpha(E)$ is sharp. This result was recently proved 
in~\cite{FMM} in the case when $K$ is a ball.

\vspace{8pt}
Finally we use Theorem  \ref{main thm 2} to prove a sharp quantitative version of the Brunn-Minkowski 
inequality \eqref{brunn min orig}  when $F$ is convex. Let us define
\begin{equation*}
\alpha(E,F):=\inf_{x\in\R^n}|(E+x)\triangle sF|, \;\text{ where } s=(|E|/|F|)^{1/n}.
\end{equation*}

\begin{corollary} \label{corollary}
Let $E,F\subset\R^n$ be two measurable, bounded and not empty sets. Assume that $F$ is convex.
Then, 
\begin{equation}\label{Brunn-Minkowski}
|E+F|^{1/n}-|E|^{1/n} - |F|^{1/n}\geq c_n\,\min\{|E|,|F|\}^{1/n}\,\frac{\alpha^2(E,F)}{|E|^2}.
\end{equation}
\end{corollary}

\vspace{8pt}

This result was proved in \cite[Theorem 1.2]{FigMP} (see also \cite[Theorem 1]{FigMP09}) in the case when both the sets
are  convex. In  \cite[Theorem 1.1]{CM16} the above result was proved with $\alpha^4(E,F)$.
For two general sets the best result to date  has been provided in~\cite{FJ} (see also \cite{FJ2}), 
but it is not known if the  exponent on the asymmetry is optimal. The sharp stability of the Brunn-Minkowski inequality 
for general sets is one of the main open problems in the field. Also the optimal dimensional dependence in inequalities
\eqref{conce convex} and \eqref{Brunn-Minkowski} is not known (see Remark \ref{dimensional dependence}).	


\section{The Gaussian concentration}

\noindent
In this section we provide a proof of Theorem \ref{main theorem}.
The symbol $c$ will denote a positive absolute constant,
whose value is not specified and may vary from line to line.

Let us recall the definition and some basic results for the Gaussian perimeter. 
For an introduction to sets of finite perimeter we refer to \cite{Ma}. 
If $E$ is a set of locally finite perimeter, its  Gaussian perimeter is defined as
\begin{equation*}
P_\gamma(E) := \frac{1}{(2\pi)^{\frac{n-1}{2}}}\int_{\partial^* E}e^{-\frac{|x|^2}{2}}d\Ha^{n-1}(x),
\end{equation*}
where $\Ha^{n-1}$ is the $(n-1)$-dimensional Hausdorff measure and $\partial^* E$
is the reduced boundary of~$E$. If $E$ is an open set with Lipschitz boundary, then 
\begin{equation} \label{derivative 1}
P_\gamma(E)= \sqrt{2\pi}\,\lim_{r\rightarrow0^+}\frac{\gamma(E+B_r)-\gamma(E)}{r}.
\end{equation}
In particular, from the concentration inequality \eqref{concentration ine} one obtains the Gaussian isoperimetric inequality:
given an open set $E$ with measure $\gamma(E)=\phi(s)$, 
\begin{equation*}
P_\gamma(E)\geq e^{-\frac{s^2}{2}},
\end{equation*}
and the equality holds if and only if $E=H_{\unit,s}$ for some $\unit \in\Sf^{n-1}$ \cite{CK}.
On the other hand, it is not difficult to see  that the  isoperimetric inequality implies the concentration inequality \eqref{concentration ine}. 
Our proof of Theorem  \ref{main theorem} is based on the robust version of the Gaussian isoperimetric inequality:
for every $s \in \R$ and for every set $E\subset\R^n$  of locally finite perimeter with $\gamma(E)=\phi(s)$ it holds
\begin{equation}\label{quanti standard new}
P_\gamma(E)-e^{-\frac{s^2}{2}}\geq c\,\frac{e^{\frac{s^2}{2}}}{1+s^2}\,\alpha_\gamma^2(E).
\end{equation}
This estimate has been recently proved in \cite[Corollary 1]{BBJ} (see also \cite{CFMP,EL,MN,MN2}). 
Note that letting $r \to 0$ in \eqref{main estimate} we obtain \eqref{quanti standard new}  by \eqref{derivative 1}  
(with a slightly worse dependence on $s$).  Therefore since the exponent on the Fraenkel asymmetry in 
\eqref{quanti standard new} is sharp (see \cite{CFMP}), also the exponent in \eqref{main estimate} is sharp. 

First we need a simple lemma, which proof is a modification of \cite[Lemma 2.1]{CM16}. 
\begin{lemma} \label{lemma gauss 0}
Let $r>0$ and let  $E$ be a measurable set. Then 
\begin{equation*}
\gamma(E+B_r)\geq\gamma(E) + \frac{1}{\sqrt{2\pi}}\int_0^r P_\gamma(E+B_\rho)\, d\rho.
\end{equation*}
\end{lemma}

\begin{figure}
\centering
\includegraphics[width=10cm]{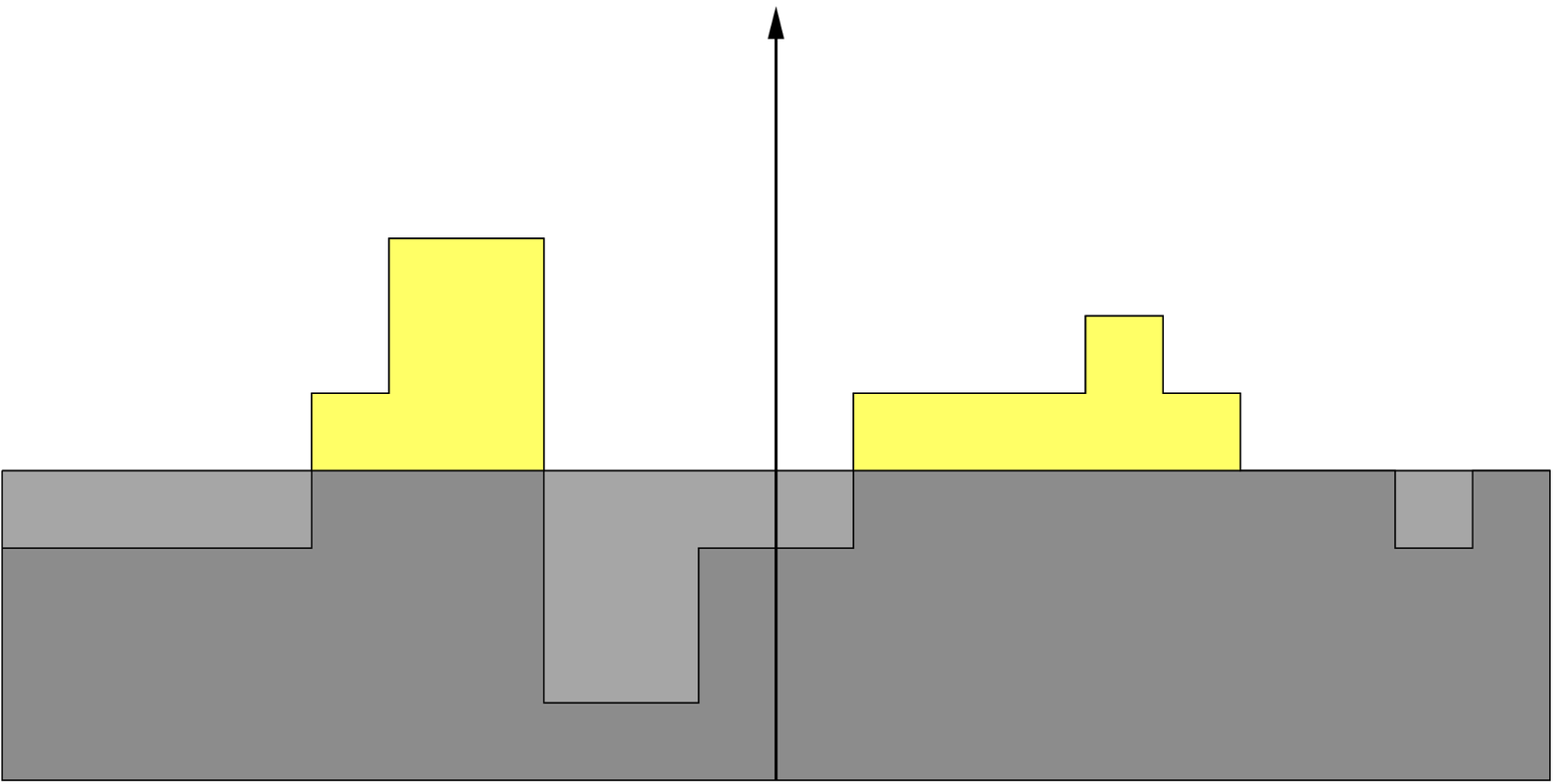}
\caption{In dark gray the set $E\cap H_{\unit,s}$, in light gray the set $H_{\unit,s}\setminus E$, 
and in yellow the set $E\setminus H_{\unit,s}$.}
\label{fig1}

\vspace{10pt}
\centering
\includegraphics[width=10cm]{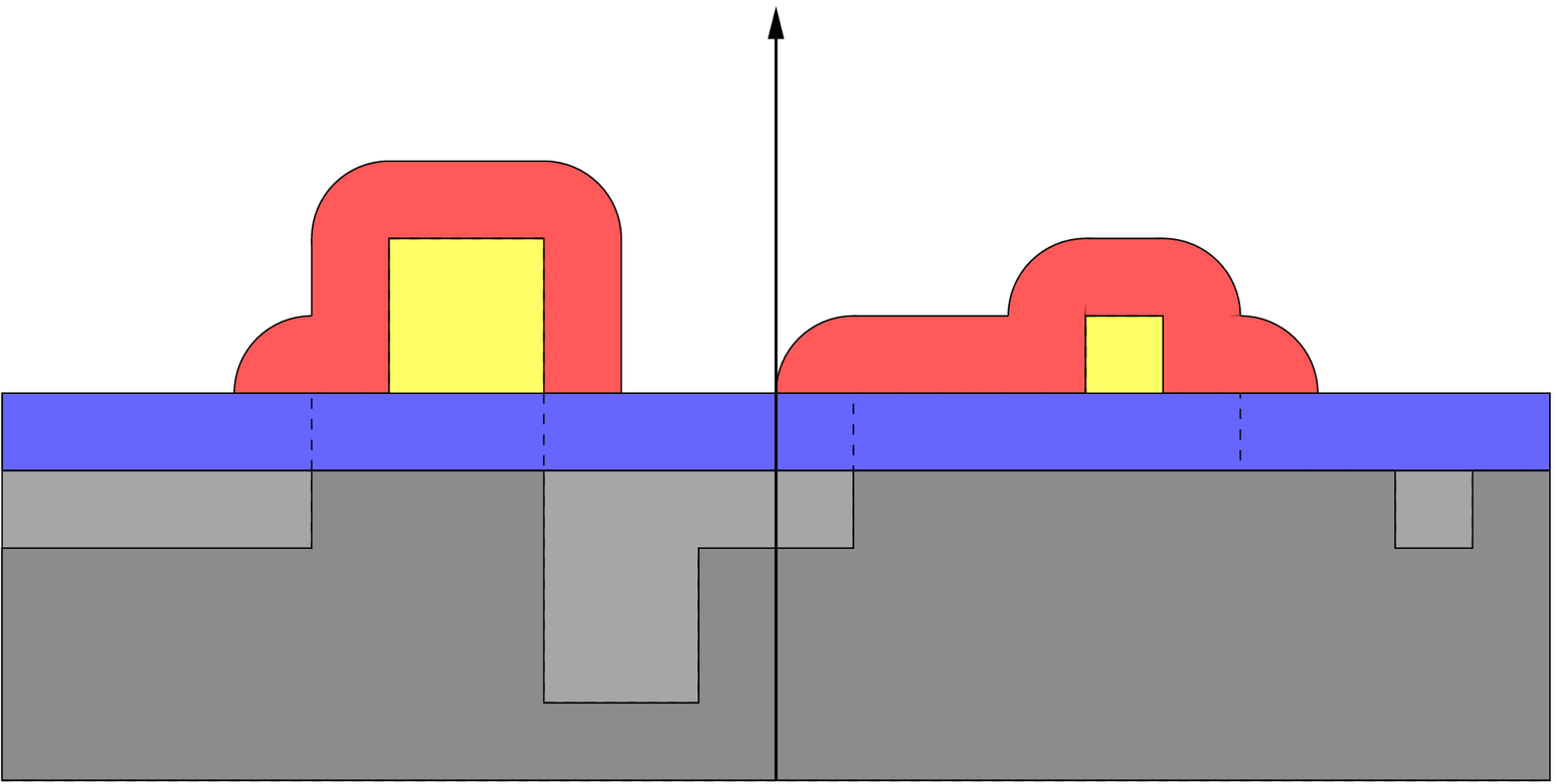}
\caption{In blue the set $H_{\unit,s+r}\setminus H_{\unit,s}$, 
and in red and yellow the set $(E+B_r)\setminus H_{\unit,s+r}$.}
\label{fig2}
\end{figure}

\vspace{6pt}
We need also a second lemma, which is a crucial point in the proof of Theorem \ref{main theorem}.
The lemma states that the $r$-enlargement of a half-space $H_s$ cannot completely cover the $r$-enlargement of the set $E$,
as depicted in Figures~\ref{fig1}-\ref{fig2}. This simple geometric fact is essentially due to the convexity of the half-space 
and therefore it is not surprising that a similar result holds also in the Euclidian case for the enlargement with a given 
convex set $K$ (see Lemma \ref{lemma eucl 1} in the next section).

\begin{lemma}\label{lemma gauss}
For every $\unit\in\Sf^{n-1}$, $s\in\R$, $r\in(0,1]$, and for every subset $E \subset \R^n$ 
 such that  $\gamma(E) = \phi(s)$ the following estimate holds:
\begin{equation*}
\gamma((E+B_r) \setminus H_{\unit,s+r}) \geq \frac{e^{-s^+}}{5}\,\gamma(E \setminus H_{\unit,s}).
\end{equation*}

\vspace{3pt}
\noindent
Here $s^+=\max\{s,0\}$. 
\end{lemma}
\begin{proof}
We split the set $E$ in two parts: 
$E^+:=\{x\in E \colon \la x,\unit\ra\geq s+r \}$ and $E^-:=E\setminus E^+$.
Of course 
\begin{equation}\label{lemma 1-1}
(E^++B_r)\setminus H_{\unit,s+r}\supset E^+\setminus H_{\unit,s+r}=E^+.
\end{equation}

\vspace{3pt}
\noindent
On the other hand, $E^-+B_r\supset E^-+r'\unit$ for any $r'\in(0,r)$ and 
\begin{equation*}\begin{split}
\gamma((E^-+r'\unit)\setminus H_{\unit,s+r'})
&=\frac{1}{(2\pi)^{\frac{n}{2}}}\int_{(E^-+r'\unit)\setminus H_{\unit,s+r'}} e^{-\frac{|x|^2}{2}}dx\\
&=\frac{1}{(2\pi)^{\frac{n}{2}}}\int_{(E^-)\setminus H_{\unit,s}} e^{-\frac{|x+r'\unit|^2}{2}}dx\\
&\geq\frac{1}{(2\pi)^{\frac{n}{2}}}e^{-\frac{(3+2s^+)}{2}}\int_{(E^-)\setminus H_{\unit,s}} e^{-\frac{|x|^2}{2}}dx\\
&\geq \frac{e^{-s^+}}{5}\,\gamma(E^-\setminus H_{\unit,s}),
\end{split}\end{equation*}
since $|x+r'\unit|^2\leq |x|^2+3+2s^+$ in $E^-$. 
Therefore, letting $r' \to r$, we have
\begin{equation}\label{lemma 1-2}
\gamma((E^-+B_r)\setminus H_{\unit,s+r})\geq \frac{e^{-s^+}}{5}\,\gamma(E^-\setminus H_{\unit,s}).
\end{equation}
Finally, from \eqref{lemma 1-1} and \eqref{lemma 1-2} we get
\begin{equation*}\begin{split}
\gamma((E+B_r) \setminus H_{\unit,s+r})
=&\gamma((E^++B_r) \setminus H_{\unit,s+r}) + \gamma((E^-+B_r) \setminus H_{\unit,s+r})\\
\geq&\gamma(E^+)+\frac{e^{-s^+}}{5}\,\gamma(E^-\setminus H_{\unit,s})\\
\geq&\frac{e^{-s^+}}{5}\,\gamma(E\setminus H_{\unit,s}).
\end{split}\end{equation*}
\end{proof}

\vspace{4pt}
\begin{proof}[\textbf{Proof of Theorem \ref{main theorem}}]
 Let us first show that we may assume 
\begin{equation}\label{may assume}
 \gamma(E+B_r) \leq \phi(s+r+1).
\end{equation}
To this aim we first estimate 
\begin{equation}\label{bound on rho 1}
\phi(s+r+1) - \phi(s+r) = \frac{1}{\sqrt{2\pi}} \int_{s+r}^{s+r+1} e^{-\frac{t^2}{2}} \, dt \geq \frac{1}{\sqrt{2\pi}} e^{-\frac{(|s|+r+1)^2}{2}}.
\end{equation}
On the other hand 
\begin{equation}\label{bound on rho 2}
\begin{split}
\frac{\alpha_\gamma(E)}{2} \leq \phi(-|s|) &=  \frac{1}{\sqrt{2\pi}} \int_{|s|}^{\infty} e^{-\frac{t^2}{2}} \, dt \\
&= \frac{1}{\sqrt{2\pi}}\left( \int_{|s|}^{|s|+1} e^{-\frac{t^2}{2}} \, dt  + \int_{|s|+1}^{\infty} e^{-\frac{t^2}{2}}dt   \right)\\
&\leq \frac{1}{\sqrt{2\pi}}\left( e^{-\frac{s^2}{2}}  + \int_{|s|+1}^{\infty} te^{-\frac{t^2}{2}}dt   \right) \\
&\leq \frac{1}{\sqrt{2\pi}}\left( e^{-\frac{s^2}{2}} +   e^{-\frac{(|s|+1)^2}{2}} \right) \leq \frac{2}{\sqrt{2\pi}} e^{-\frac{s^2}{2}}.
\end{split}
\end{equation}
Assume now that  \eqref{may assume} does not hold. Then we have by \eqref{bound on rho 1} and \eqref{bound on rho 2} that 
\[
\begin{split}
 \gamma(E+B_r) - \phi(s+r) &\geq  \phi(s+r+1) - \phi(s+r) \\ 
&\geq  \frac{1}{\sqrt{2\pi}} e^{-\frac{(|s|+r+1)^2}{2}} \\
 &\geq c \, e^{s^2}  e^{-\frac{(|s|+r+1)^2}{2}} \alpha_\gamma^2(E).
\end{split}
\]
Hence if \eqref{may assume}  does not hold  then  \eqref{main estimate} is true. 

Because of the non-monotonicity of the quantity $\gamma(E+B_r)-\phi(s+r)$, 
we have to divide the rest of the proof in several steps.  We first prove the theorem when  $r\in(0,1]$. We divide this part of the proof in two cases.

\noindent \textbf{Case 1.}  We first assume  that for all $\rho\in(0,r]$ it holds
\begin{equation}\label{dicotomy}
\gamma(E+B_\rho)-\phi(s+\rho) \leq \eps\,e^{\frac{s^2}{2}}e^{-3|s|}\alpha_\gamma^2(E),
\end{equation}
where $\eps>0$ is a small number to be chosen later.

We define an auxiliary function $f:(0,r)\to\R$ by 
\begin{equation*}
f(\rho) := P_\gamma(E+ B_\rho) - e^{-(s+\rho)^2/2}.
\end{equation*}
Then by Lemma \ref{lemma gauss 0}
\begin{equation*}
\gamma(E+ B_r)-\phi(s+r) \geq \frac{1}{\sqrt{2\pi}}\int_0^rf(\rho)d\rho.
\end{equation*}
Therefore in order to prove \eqref{main estimate}, it is  enough to estimate $f(\rho)$. Let us fix $\rho \in (0,r)$.
Let $\hat{\rho} >0$ be such that $\gamma(E+ B_\rho)=\phi(s+\hat{\rho})$. 
Note that by the concentration inequality $\hat{\rho}\geq \rho$.  Moreover  \eqref{may assume} implies that $\hat{\rho}\leq 2$.

Let $\unit_\rho\in\Sf^{n-1}$ be a direction that realizes 
$\min_{\unit\in\Sf^{n-1}}\gamma((E+B_\rho)\triangle H_{\unit,s+\hat{\rho}})$, and let $s^-=-\min\{s,0\}$.
By the stability of the Gaussian isoperimetric inequality \eqref{quanti standard new} 
and by Lemma \ref{lemma gauss} we have 
\begin{equation*}\begin{split}
f(\rho) &= \left(P_\gamma(E+ B_\rho) - e^{-\frac{(s+\hat{\rho})^2}{2}} \right) 
+ e^{-\frac{(s+\hat{\rho})^2}{2}} - e^{-\frac{(s+\rho)^2}{2}}\\
&\geq c\,\frac{e^{\frac{(s+\hat{\rho})^2}{2}}}{1+(s+\hat{\rho})^2}\,\gamma((E+B_\rho)\triangle H_{\unit_\rho,s+\hat{\rho}})^2 
+e^{-\frac{(s+\hat{\rho})^2}{2}} - e^{-\frac{(s+\rho)^2}{2}}\\
&\geq c\, \frac{e^{\frac{s^2}{2}}e^{-2s^-}}{(1 + s^2)}\, \bigl(\gamma((E+B_\rho)\setminus H_{\unit_\rho,s+\rho})
-\gamma(H_{\unit_\rho,s+\hat{\rho}}\setminus H_{\unit_\rho,s+\rho})\bigr)^2
+e^{-\frac{(s+\hat{\rho})^2}{2}} - e^{-\frac{(s+\rho)^2}{2}}\\
&\geq c\, \frac{e^{\frac{s^2}{2}}e^{-2s^-}}{(1 + s^2)}\, \Bigl(\frac{e^{-s^+}}{10}\alpha_\gamma(E)
-\gamma(H_{\unit_\rho,s+\hat{\rho}}\setminus H_{\unit_\rho,s+\rho}) \Bigr)^2
+e^{-\frac{(s+\hat{\rho})^2}{2}} - e^{-\frac{(s+\rho)^2}{2}}.
\end{split}\end{equation*}
By the definition of $\hat \rho$, by \eqref{dicotomy} and by  \eqref{bound on rho 2} we have 
\[
\begin{split}
\gamma(H_{\unit_\rho,s+\hat{\rho}}\setminus H_{\unit_\rho,s+\rho}) &= \gamma(E + B_\rho) - \phi(s+ \rho)\\
&\leq \eps\,e^{\frac{s^2}{2}}e^{-3|s|}\alpha_\gamma^2(E) \leq\frac{e^{-3|s|}}{20}\alpha_\gamma(E)
\end{split}
\]
when $\eps$ is small enough. We use  \eqref{dicotomy} to estimate 
\begin{equation*}\begin{split} 
e^{-\frac{(s+\hat{\rho})^2}{2}} - e^{-\frac{(s+\rho)^2}{2}}
&=-\int_{s+\rho}^{s+\hat{\rho}}te^{-\frac{t^2}{2}}dt
\geq-(|s|+2)\int_{s+\rho}^{s+\hat{\rho}} e^{-\frac{t^2}{2}}dt\\
&= -\sqrt{2\pi}\,(|s|+2) \left(\gamma(E+B_\rho)-\phi(s+\rho) \right) \\
&\geq-\eps\,\sqrt{2\pi}\,(|s|+2)  e^{\frac{s^2}{2}}e^{-3|s|}\alpha_\gamma^2(E).
\end{split}\end{equation*}
Therefore by the previous three estimates we have 
\begin{equation*}\begin{split}
f(\rho)  &\geq  c\, \frac{e^{\frac{s^2}{2}}e^{-2|s|}}{(1 + s^2)}\, \alpha_\gamma^2(E) 
-\eps\,\sqrt{2\pi}\,(|s|+2)  e^{\frac{s^2}{2}}e^{-3|s|}\alpha_\gamma^2(E).\\
&\geq c\, e^{\frac{s^2}{2}}e^{-3|s|}\, \alpha_\gamma^2(E) 
\end{split}\end{equation*}
when $\eps$ is small enough. Thus we have the claim  \eqref{main estimate}  in this case. 

\medskip
\noindent \textbf{Case 2.}
In this case we assume that there is $\rho \in (0,r]$ such that 
\begin{equation}\label{dicotomy 2}
\gamma(E+B_\rho)-\phi(s+\rho) \geq \eps\,e^{\frac{s^2}{2}}e^{-3|s|}\alpha_\gamma^2(E).
\end{equation}

Let $\hat \rho >0$ be such that 
\[
\gamma(E+B_\rho) = \phi(s+\hat \rho).
\]
The concentration inequality implies 
\begin{equation}\label{case 2 concentration}
\gamma(E+B_r) \geq \phi(s+\hat \rho +r - \rho).
\end{equation}
Note that  \eqref{may assume}  gives $\hat \rho \leq 2$. We may therefore estimate 
\[
\begin{split}
\phi(s+\hat \rho +r - \rho) - \phi(s+r) &= \frac{1}{\sqrt{2 \pi}} \int_{s +r }^{s+\hat \rho +r - \rho} e^{-\frac{t^2}{2}} \, dt =  \frac{1}{\sqrt{2 \pi}} \int_{s +\rho}^{s+\hat \rho} e^{-\frac{(t+r-\rho)^2}{2}} \, dt \\
&\geq c \, e^{-s^+} \int_{s +\rho }^{s+\hat \rho } e^{-\frac{t^2}{2}} \, dt = c \, e^{-s^+} (\phi(s+\hat \rho) - \phi(s+\rho)).
\end{split}
\]
We deduce from  \eqref{case 2 concentration}, from the definition of $\hat \rho$ and from \eqref{dicotomy 2} that
\[
\gamma(E+B_r) - \phi(s+r) \geq c \, e^{-s^+} (\gamma(E+B_\rho)-\phi(s+\rho) ) 
\geq c \, \eps \, e^{\frac{s^2}{2}}e^{-4|s|}\alpha_\gamma^2(E),
\]
which proves the claim  \eqref{main estimate}.  

\medskip

We are left to prove  the claim \eqref{main estimate} when $r >1$.  Since we have already proved the result for $r=1$ we have that 
\[
\gamma(E+B_1) - \phi(s+1) \geq c_1 \, e^{\frac{s^2}{2}}e^{-\frac{(|s|+5)^2}{2}}\alpha_\gamma^2(E),
\]
for an absolute constant $c_1>0$.
The rest of the proof is the same as in the Case 2 above.  Let $\hat \rho \geq 1$ be such that 
\[
\gamma(E+B_1) = \phi(s+\hat \rho).
\]
The concentration inequality implies 
\[
\gamma(E+B_r) \geq \phi(s+\hat \rho +r - 1).
\]
Note that  \eqref{may assume} and the above inequality give $\hat \rho \leq 2$ and we may  estimate as before
\[
\begin{split}
\phi(s+\hat \rho +r - 1)- \phi(s+r) &= \frac{1}{\sqrt{2 \pi}} \int_{s +r }^{s+\hat \rho +r - 1} e^{-\frac{t^2}{2}} \, dt \\
&\geq \frac{1}{\sqrt{2 \pi}}\, e^{-(|s|+1)(r-1)} \, e^{- \frac{r^2}{2}} \int_{s +1 }^{s+\hat \rho } e^{-\frac{t^2}{2}} \, dt \\
&= e^{-(|s|+1)(r-1)} \, e^{- \frac{r^2}{2}} (\phi(s+\hat \rho) - \phi(s+1)).
\end{split}
\]
The four estimates above yield
\[
\gamma(E+B_r) - \phi(s+r) \geq c_1 \, e^{s^2} e^{-\frac{(|s| + 5)^2}{2}} e^{-(|s|+1)(r-1)} e^{-\frac{r^2}{2}} \, \alpha_\gamma^2(E)
\]
and the claim  \eqref{main estimate} follows.
\end{proof}

\begin{remark}
The best known value for the constant $c$ in the isoperimetric estimate \eqref{quanti standard new} is 
$c_{\text{iso}}=1/(48\sqrt{2\pi})$, obtained by a slightly refinement of the argument in \cite{BBJ}. 
A careful study of the proof of Theorem \ref{main theorem}
shows that one can take the value for the constant $c$ in \eqref{main estimate} to be $c_{\text{iso}}/2000$. 
In both the cases the values are not optimal. 
\end{remark}


\section{The Euclidean concentration}

\noindent
In this section we provide a proof of Theorem \ref{main thm 2}.
The symbol $c_n$ will denote a positive constant depending on $n$, whose value is 
not specified and which may vary from line to line.

The proof of Theorem~\ref{main thm 2} is 
based on the quantitative Wulff inequality  provided in \cite{FigMP}. 
Let us briefly introduce some notation. We set
\begin{equation*}
||\nu||_* =  \sup_{x \in K} \la x, \nu \ra, \qquad \nu \in \Sf^{n-1}
\end{equation*}
and define the anisotropic perimeter for a set $E$ with locally finite perimeter as
\begin{equation*}
P_K(E) = \int_{\partial^* E} ||\nu||_* \, d\Ha^{n-1},
\end{equation*}
where $\partial^* E$ denotes the reduced boundary of $E$. When  $E$ is  open with Lipschitz boundary we have
\begin{equation*}
P_K(E) =  \lim_{r \to 0} \frac{|E+ rK| - |E|}{r}.
\end{equation*}    
The result in \cite{FigMP} states that for every set $E$ of locally finite perimeter with $|E| = |sK|$ it holds
\begin{equation} \label{FigMP} 
P_K(E) - P_K(sK) \geq \frac{c_n}{|K|} \frac{1}{s^{n+1}} \alpha^2(E).
\end{equation}

The following lemma  is the counterpart of Lemma \ref{lemma gauss 0} in the Euclidean case. 
\begin{lemma} \label{lemma eucl 0}
Let $r>0$ and let  $E$ be a measurable set such that $|E+rK| < \infty$. Then 
\begin{equation*}
|E+rK|\geq |E| + \int_0^r P_K(E + \rho K)\, d\rho.
\end{equation*}
\end{lemma}

Similarly to Lemma \ref{lemma gauss}, the measure of $E \setminus K$ is increasing along the growth.
\begin{lemma}\label{lemma eucl 1}
Let $E \subset \R^n$ be a measurable set such that $|E|= |K|$.  Then for every $r>0$ it holds
\begin{equation*}
|(E+rK) \setminus (1+r)K| \geq |E \setminus K|.
\end{equation*}
\end{lemma}

\begin{proof}
Let $r_0\geq 1$ be such that $|E \cup K| = |r_0 K|$. Then 
\begin{equation} \label{symm diff}
|E\setminus K| = |E\cup K| - |K| = (r_0^n -1)|K|. 
\end{equation}
By the concentration inequality \eqref{aniso concentration}  it holds 
\[
|(E\cup K) + rK| \geq |(r_0+r)K|.
\]
Since $[(E+rK)\setminus(1+r)K]\cup(1+r)K=(E\cup K)+rK$, one has
\[
|(E\cup K) + rK| - |(1+r)K| = |(E+rK) \setminus (1+r)K|. 
\]
These together yield 
\[
\begin{split}
|(E+rK) \setminus (1+r)K| &\geq |(r_0+r)K| - |(1+r)K| \\
&= ((r+r_0)^n- (1+r)^n)|K| \geq  (r_0^n-1)|K|,
\end{split}
\]
where the last inequality follows from the fact that $t \mapsto ((t+r_0)^n- (1+t)^n)$ is nondecreasing. 
The claim then follows from \eqref{symm diff}. 
\end{proof}

\begin{proof}[\textbf{Proof of Theorem \ref{main thm 2}}]
By scaling we may assume that $|K|=1$. Let us first prove the claim when $r \in (0,1]$.  
We may assume that $|E+ rK|\leq |3K|$. Indeed if $|E+ rK|> |3K|$ then
\begin{equation*}
|E+rK|-|(1+r)K|>|3K|-|2K|=3^n-2^n\geq \frac{3^n-2^n}{4}\alpha^2(E).
\end{equation*}

Let us define $f:(0,r) \to \R$,
\begin{equation*}
f(\rho) = P_K(E+ \rho K) - P_K((1+ \rho)K).
\end{equation*}
By Lemma \ref{lemma eucl 0} we have that $f(\rho) < \infty$ for almost every $\rho$. 
By the concentration inequality \eqref{aniso concentration} it holds $|E+ \rho K|\geq |(1+ \rho)K|$ 
for every $\rho \in (0,r)$. Let us fix $\rho < r$ and let $\hat \rho \geq \rho$ 
be such that $|(1+ \hat \rho)K| =|E+ \rho K|$. 
Then by $|E+ rK|\leq |3K|$ we have $\hat \rho \leq 2$.  By the stability of the Wulff  inequality \eqref{FigMP} 
and recalling that $P_K(\lambda K)=n\lambda^{n-1}|K|$ for every $\lambda>0$, we have
\begin{equation*}
\begin{split}
f(\rho) &= \left[P_K(E+ \rho K) - P_K((1+ \hat \rho)K) \right]+  \left[P_K((1+ \hat \rho)K) - P_K((1+ \rho)K) \right]\\
&\geq c_n \Bigl(  \inf_{x \in \R^n } |((E+ \rho K) +x)\triangle (1+ \hat \rho)K|\Bigr)^2 
+  n \left((1+ \hat \rho)^{n-1} - (1+ \rho)^{n-1}\right).
\end{split}
\end{equation*}
We estimate the last term by
\begin{equation*}
(1+ \hat \rho)^{n-1} - (1+ \rho)^{n-1}
\geq\frac{(1+\hat{\rho})^n - (1+\rho)^n}{2+\rho+\hat{\rho}}
\geq \frac{1}{5}\left(|(1+ \hat \rho)K| - |(1+ \rho)K|\right).
\end{equation*}
Thus  it holds 
\begin{equation*}
\begin{split}
f(\rho) &\geq c_n \left(  \inf_{x \in \R^n } |((E+ \rho K) +x)\triangle (1+ \hat \rho)K|\right)^2  +  \frac{n}{5}\left(|(1+ \hat \rho)K| - |(1+ \rho)K|\right) \\
&\geq c_n\left(  \inf_{x \in \R^n } |((E+ \rho K) +x)\setminus (1+ \hat \rho)K|\right)^2 + c_n \left(|(1+ \hat \rho)K| - |(1+ \rho)K|\right)^2 \\
&\geq c_n \left(  \inf_{x \in \R^n }|((E+ \rho K) +x)\setminus (1+ \rho)K|\right)^2,
\end{split}
\end{equation*}
where  the last inequality is a simple consequence of 
\begin{equation*}
|((E+ \rho K) +x)\setminus (1+ \rho)K | \leq |((E+ \rho K) +x) \setminus  (1+ \hat \rho)K| 
+ | (1+ \hat \rho)K \setminus (1+ \rho)K|.
\end{equation*}
Now we use Lemma \ref{lemma eucl 1} and deduce that for every $\rho \in (0,r)$ it holds  
\begin{equation*}
 \inf_{x \in \R^n } |(E+ \rho K) +x)\setminus (1+ \rho)K | \geq \frac{\alpha(E)}{2}.
\end{equation*} 
Hence we conclude that 
\begin{equation*}
f(\rho) \geq c_n \,    \alpha^2(E)
\end{equation*}
for every $\rho \in (0,r)$ and the result follows by Lemma  \ref{lemma eucl 0}. 

Let us  assume $r >1$. By the previous argument we have 
\begin{equation}\label{ascendent step}
|E+ K| - |2K| \geq c_n\, \alpha^2(E). 
\end{equation}
By choosing $c_n\leq1/4$ we may assume that  $c_n\,\alpha^2(E)\leq1$. Thus we have the elementary inequality $2^n+c_n\, \alpha^2(E)\geq (2+c_n\, \alpha^2(E)/(n3^n))^n$.
Therefore, by \eqref{ascendent step} we have $|E+ K|  \geq |(2+c_n\,\alpha^2(E))K|$ for a  possibly smaller $c_n$. 
Using this estimate and the concentration inequality \eqref{aniso concentration}   we get
\begin{equation*}
 |(E+ K) + (r-1)K| \geq |(1+r+c_n\,\alpha^2(E))K|.
\end{equation*}
Thus we conclude that 
\begin{equation*}
|E+ rK| - |(1+r)K| \geq |(1+r+c_n\,\alpha^2(E))K| - |(1+r)K| \geq c_n (1+r)^{n-1} \alpha^2(E).
\end{equation*}
\end{proof}

\vspace{-5pt}
\begin{remark}\label{dimensional dependence}
The constant $c_n$ in the isoperimetric estimate \eqref{FigMP} decays at most like $n^{-13}$ as $n\to\infty$. 
Instead, a careful study of the proof of Theorem \ref{main thm 2} shows that the value 
for the constant $c_n$ in \eqref{conce convex} decays at most like $9^{-n}$. 
Both these decays do not seem optimal. In particular, for \eqref{FigMP} we conjectured that the constant $c_n$ is 
in fact independent of the dimension.
\end{remark}

\begin{proof}[\textbf{Proof of Corollary \ref{corollary}}]
We may  assume that $|E+F|\leq 3^n \max\{|E|, |F|\}$, since otherwise \eqref{Brunn-Minkowski} is trivially true.
Let $s >0$ be such that  $|E|=|sF|$. The concentration inequality  \eqref{aniso concentration} implies $|E+F|\geq|(1+s)F|$. 
Therefore,
\begin{equation*}\begin{split}
|E+F|^{1/n}-|E|^{1/n} - |F|^{1/n}
&=|E+F|^{1/n}-|(1+s)F|^{1/n}\\
&\geq\frac{1}{n|E+F|^{(n-1)/n}} [|E+F|-|(1+s)F|]\\
&\geq c_n \left(\frac{1} {\max\{|E|, |F|\}} \right)^{(n-1)/n} [|E+F|-|(1+s)F|].
\end{split}\end{equation*}
Assume first $|E|\geq|F|$. By  \eqref{conce convex} with $K=sF$ and $r=1/s$ we get,
\begin{equation*}\begin{split}
|E+F|^{1/n}-|E|^{1/n} - |F|^{1/n}
&\geq \frac{c_n}{s |E|^{(n-1)/n}}\,\frac{\alpha^2(E,F)}{|E|}\\
&= c_n\,|F|^{1/n}\,\frac{\alpha^2(E,F)}{|E|^2}.
\end{split}\end{equation*}
On the other hand when  $|F|\geq|E|$, the same argument yields
\begin{equation*}\begin{split}
|E+F|^{1/n}-|E|^{1/n} - |F|^{1/n}
&\geq \frac{c_n}{s^{n-1} |F|^{(n-1)/n}}\,\frac{\alpha^2(E,F)}{|E|}\\
&=c_n\,|E|^{1/n}\,\frac{\alpha^2(E,F)}{|E|^2}.
\end{split}\end{equation*}
\end{proof}


\end{document}